\documentclass{article}
\usepackage[final]{graphicx}
\usepackage{psfrag}
\usepackage{amsmath,amsfonts,amssymb,amsxtra,subeqnarray}

\newlength{\extramargin}
\newcommand {\Real}{\ensuremath{{\mathbb{R}}}}

\newcommand{\setS}{\ensuremath{\mathcal S}}

\renewcommand{\L}{\ensuremath{\mathcal L}}

\newcommand{\yu}{\ensuremath{{\mathbf{u}}}}
\newcommand{\vi}{\ensuremath{{\mathbf{v}}}}
\newcommand{\ex}{\ensuremath{{\mathbf{x}}}}
\newcommand{\exr}{\ensuremath{{\mathbf{x}}_{\rm r}}}
\newcommand{\er}{\ensuremath{{\mathbf{e}}_{\rm r}}}

\newcommand{\vay}{\ensuremath{{\mathbf{y}}}}

\newcommand{\dabilyu}{\ensuremath{{\mathbf{w}}}}

\newcommand{\Bbig}{\ensuremath{{\mathbf{B}}}}
\newcommand{\Wbig}{\ensuremath{{\mathbf{W}}}}
\newcommand{\Rbig}{\ensuremath{{\mathbf{R}}}}
\newcommand{\Qbig}{\ensuremath{{\mathbf{Q}}}}
\newcommand{\Kbig}{\ensuremath{{\mathbf{K}}}}
\newcommand{\Lbig}{\ensuremath{{\mathbf{L}}}}
\newcommand{\Dbig}{\ensuremath{{\mathbf{D}}}}
\newcommand{\Hbig}{\ensuremath{{\mathbf{H}}}}
\newcommand{\Sbig}{\ensuremath{{\mathbf{S}}}}
\newcommand{\Cbig}{\ensuremath{{\mathbf{C}}}}
\newcommand{\Bbigr}{\ensuremath{{\mathbf{B}}_{\rm r}}}
\newcommand{\Cbigr}{\ensuremath{{\mathbf{C}}_{\rm r}}}
\newcommand{\Lbigr}{\ensuremath{{\mathbf{L}}_{\rm r}}}
\newcommand{\Kbigr}{\ensuremath{{\mathbf{K}}_{\rm r}}}
\newcommand{\Abig}{\ensuremath{{\mathbf{A}}}}
\newcommand{\Abigr}{\ensuremath{{\mathbf{A}}_{\rm r}}}

\newtheorem{theorem}{Theorem}

\newtheorem{lemma}{Lemma}

\newtheorem{definition}{Definition}

\newtheorem{proposition}{Proposition}
\newenvironment{proof}{\noindent {\bf Proof.}}{\hfill \hspace*{1pt}\hfill$\blacksquare$}

\begin{document}
\title{Synchronization of linear systems via relative actuation}
\author{S. Emre Tuna\footnote{The author is with Department of
Electrical and Electronics Engineering, Middle East Technical
University, 06800 Ankara, Turkey. The work has been completed
during his sabbatical stay at Department of Electrical and Electronic
Engineering, The University of Melbourne, Victoria 3010,
Australia. Email: {\tt etuna@metu.edu.tr}}} \maketitle

\begin{abstract}
Synchronization in networks of discrete-time linear time-invariant
systems is considered under relative actuation. Neither input nor
output matrices are assumed to be commensurable. A distributed
algorithm that ensures synchronization via dynamic relative output
feedback is presented.
\end{abstract}

\section{Introduction}

Theory on synchronization in networks of linear systems with
general dynamics has reached a certain maturity over the last
decade; see, for instance,
\cite{scardovi09,seo09,li10,yang11,cao13}. A significant part of
this theory is founded on the following setup. The nominal
individual agent dynamics reads ${\dot x}_{i}=Ax_{i}+Bu_{i}$ with
$y_{i}=Cx_{i}$. And, the signals available for decision are of the
form $z_{i}=\sum_{j} a_{ij}(y_{i}-y_{j})$, where the nonnegative
scalars $a_{ij}$ describe the so-called communication topology.
Within the boundary of this setup different approaches have
yielded various interesting solutions to the synchronization
problem, where the universal goal is to drive the agents' states
$x_{i}(t)$ to a common trajectory. E.g., communication delays are
considered in \cite{zhou14}, $\L_{2}$-gain output-feedback is
employed in \cite{jiao16}, distributed containment problem is
studied in \cite{li13a}. Among many other works contributing to
our wealth of knowledge are \cite{li13b} on adaptive protocols,
\cite{seo12} on switching topologies, and \cite{zhang11} on
optimal state feedback and observer design.

\begin{figure}[h]
\begin{center}
\includegraphics[scale=0.51]{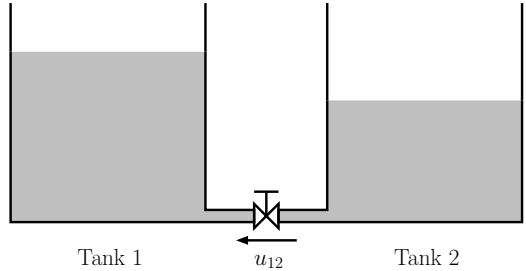}
\caption{Water tanks with a shared actuator
(pump).}\label{fig:reservoirdep}
\end{center}
\end{figure}

Despite their differences the above-mentioned works allow each
agent to have its own independent input $u_{i}$. In this paper we
shed this independence. Instead of each agent having its own input
we look at the case where each input ($u_{ij}$) is shared by a
pair of agents ($i$th and $j$th systems) in the sense displayed in
Fig.~\ref{fig:reservoirdep}. In particular, we consider the agent
dynamics ${x}_{i}^{+}=Ax_{i}+\sum_{j}B_{ij}u_{ij}$ with
$y_{ij}=C_{ij}(x_{i}-x_{j})$, where (i) the actuation is {\em
relative} (i.e., $B_{ij}u_{ij}+B_{ji}u_{ji}=0$) and (ii) the
signals available for decision read $z_{i}=\sum_{j}
C_{ij}^{T}y_{ij}$. In our setup the input matrices $B_{ij}$ are
allowed to be {\em incommensurable} in the sense that there need
not exist a common $B$ satisfying $B_{ij}=a_{ij}B$ with scalar
$a_{ij}$. In fact, two input matrices do not even have to be of
the same size. The same goes for the output matrices $C_{ij}$. The
problem we study here is that of decentralized
stabilization (of the synchronization subspace) by choosing
appropriate inputs $u_{ij}$ based on the relative measurements
$y_{ij}$. As a solution to this problem we construct a distributed
algorithm that achieves synchronization via dynamic relative output
feedback.

Let us now illustrate our setup on an example network. Consider the array of
identical electrical oscillators shown in Fig.~\ref{fig:LCarray3},
where each oscillator (of order $2p$) has $p$ nodes (excepting the
ground node) and the $k$th node of the $i$th oscillator is denoted
by $n_{k}^{(i)}$. The actuation is achieved through current
sources while the measurements are collected through voltmeters.
Each current source/voltmeter connects a pair of nodes (belonging
to two separate oscillators) with the same index number, say
$n_{k}^{(i)}$ and $n_{k}^{(j)}$. It is not difficult to see that
this architecture enjoys the form ${\dot
x}_{i}=Ax_{i}+\sum_{j}B_{ij}u_{ij}$ with
$y_{ij}=C_{ij}(x_{i}-x_{j})$ and, since each current source connects
two nodes with the same index number, the actuation throughout the network
is relative. Furthermore, the input matrices $B_{ij}$ are incommensurable. For
instance, while the current source $u_{32}$ connects $n_{1}^{(3)}$
and $n_{1}^{(2)}$, the current source $u_{12}$ connects
$n_{p}^{(1)}$ and $n_{p}^{(2)}$. Since for these two current
sources the indices (1 and $p$) of the nodes they are associated
to are different, we cannot find a scalar $a$ that satisfies
$B_{12}=aB_{32}$. Likewise, the output matrices $C_{ij}$ too are
incommensurable.

\begin{figure}[h]
\begin{center}
\includegraphics[scale=0.55]{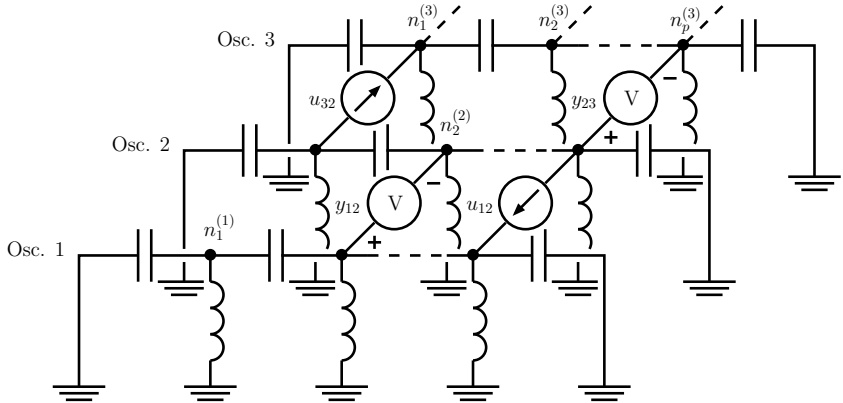}
\caption{Network of electrical oscillators.} \label{fig:LCarray3}
\end{center}
\end{figure}

We begin the remainder of the paper by providing the formal
description of the {\em array} we study. After that we present a
distributed {\em algorithm} that generates control inputs through
dynamic output feedback, followed by our main (and only) theorem,
which states that this algorithm with suitable parameter choice
achieves synchronization. To prove the theorem we first obtain the
explicit expression of the {\em closed-loop} system the array
becomes under the algorithm. Once the righthand side of the closed loop
is computed we proceed to establish {\em stability} and thus complete the
proof of the main result.

\section{Array}

Consider an {\em array} of $q$ discrete-time linear time-invariant
systems
\begin{subeqnarray}\label{eqn:array}
x_{i}^{+}&=&Ax_{i}+\sum_{j=1}^{q}B_{ij}u_{ij}\\
y_{ij}&=&C_{ij}(x_{i}-x_{j})\,;\qquad i,\,j=1,\,2,\,\ldots,\,q
\end{subeqnarray}
where $x_{i}\in\Real^{n}$ is the state of the $i$th system with
$A\in\Real^{n\times n}$, $x_{i}^{+}$ denotes the state at the next
time instant, $u_{ij}=u_{ji}\in\Real^{p_{ij}}$ is the $ij$th input
with $B_{ij}=-B_{ji}\in\Real^{n\times p_{ij}}$, and
$y_{ij}=y_{ji}\in\Real^{m_{ij}}$ is the $ij$th (relative) output
with $C_{ij}=-C_{ji}\in\Real^{m_{ij}\times n}$. We interpret the
equality $u_{ij}=u_{ji}$ as that $u_{ij}$ and $u_{ji}$ are
different notations for the same single variable. Same goes for
the oneness of $y_{ij}$ and $y_{ji}$. Note that we have to have
$B_{ii}=0$ and $C_{ii}=0$. Note also that the actuation is
relative because $B_{ij}u_{ij}+B_{ji}u_{ji}=0$. Hence the average
of the states $x_{\rm av}=q^{-1}\sum x_{i}$ evolves independently
of the inputs driving the array, i.e., we have $x_{\rm
av}^{+}=Ax_{\rm av}$. The ordered collections
$(B_{ij})_{i,j=1}^{q}$ and $(C_{ij})_{i,j=1}^{q}$ are denoted by
$(B_{::})$ and $(C_{::})$, respectively. The value of the solution
of the $i$th system at the $k$th time instant ($k=0,\,1,\,\ldots$)
is denoted by $x_{i}[k]$. The meanings of $u_{ij}[k]$ and
$y_{ij}[k]$ should be clear.

The array~\eqref{eqn:array} gives rise to the following single big
system
\begin{subeqnarray}\label{eqn:bigsystem}
\ex^{+}&=&\Abig\ex+\Bbig\yu\\
\vay&=&\Cbig\ex
\end{subeqnarray}
where $\ex=[x_{1}^{T}\ x_{2}^{T}\ \cdots\ x_{q}^{T}]^{T}$ is the
state, $\yu=[u_{12}^{T}\ u_{13}^{T}\ \cdots\ u_{1q}^{T}\,|\,
u_{23}^{T}\ u_{24}^{T}\ \cdots\
u_{2q}^{T}\,|\,\cdots\,|\,u_{(q-1)q}^{T}]^{T}$ is the input, and
$\vay=[y_{12}^{T}\ y_{13}^{T}\ \cdots\ y_{1q}^{T}\,|\, y_{23}^{T}\
y_{24}^{T}\ \cdots\ y_{2q}^{T}\,|\,\cdots\,|\,y_{(q-1)q}^{T}]^{T}$
is the output. Clearly, we have
\begin{eqnarray*}
\Abig=[I_{q}\otimes A]
\end{eqnarray*}
where $I_{q}\in\Real^{q\times q}$ is the identity matrix (which we
may also denote by $I$ when its dimension is either clear from the
context or immaterial),
\begin{eqnarray*}
\Bbig = {\rm inc}\,(B_{::}):=
\left[\begin{array}{cccc|cccc|c|c}
B_{12}&B_{13}&\cdots&B_{1q}&0&0&\cdots&0&\cdots&0\\
-B_{12}&0&\cdots&0&B_{23}&B_{24}&\cdots&B_{2q}&\cdots&0\\
0&-B_{13}&\cdots&0&-B_{23}&0&\cdots&0&\cdots&0\\
0&0&\cdots&0&0&-B_{24}&\cdots&0&\cdots&0\\
\vdots&\vdots&\ddots&\vdots&\vdots&\vdots&\ddots&\vdots&\ddots&\vdots\\
0&0&\cdots&0&0&0&\cdots&0&\cdots&B_{(q-1)q}\\
0&0&\cdots&-B_{1q}&0&0&\cdots&-B_{2q}&\cdots&-B_{(q-1)q}
\end{array}
\right]
\end{eqnarray*}
and
\begin{eqnarray*}
\Cbig=[{\rm inc}\,(C_{::}^{T})]^{T}\,.
\end{eqnarray*}
The notational choice ``inc'' has to do with that the structure of
$\Bbig$ resembles that of the {\em incidence matrix} of a graph.
Let $\setS_{n}={\rm range}\,[{\mathbf 1}_{q}\otimes I_{n}]$, where
${\mathbf 1}_{q}\in\Real^{q}$ is the vector of all ones. The set
$\setS_{n}\subset (\Real^{n})^{q}$ is called the {\em
synchronization subspace}, whose orthogonal complement, the {\em
disagreement subspace}, is denoted by $\setS_{n}^{\perp}$. Let us
construct the matrices one is all too familiar with
\begin{eqnarray*}
\Wbig_{\rm c}=[\Bbig\ \Abig\Bbig\ \cdots\ \Abig^{n-1}\Bbig]\,,\quad
\Wbig_{\rm o}=\left[\begin{array}{c}\Cbig\\ \Cbig\Abig\\ \vdots\\ \Cbig\Abig^{n-1}\end{array}\right]\,.
\end{eqnarray*}
We have ${\rm range}\,\Wbig_{\rm c}\subset\setS_{n}^{\perp}$ and
${\rm null}\,\Wbig_{\rm c}\supset\setS_{n}$ by construction. Now
we define (relative) controllability and (relative) observability
concerning the array~\eqref{eqn:array}.

\begin{definition}
The array~\eqref{eqn:array} (or the pair $[A,\,(B_{::})]$) is said
to be {\em controllable} if ${\rm range}\,\Wbig_{\rm
c}\supset\setS_{n}^{\perp}$.
\end{definition}

\begin{definition}
The array~\eqref{eqn:array} (or the pair $[(C_{::}),\,A]$) is said
to be {\em observable} if ${\rm null}\,\Wbig_{\rm
o}\subset\setS_{n}$.
\end{definition}
Note that $[A,\,(B_{::})]$ is controllable if and only if
$[(B_{::}^{T}),\,A^{T}]$ is observable. Necessary and sufficient
conditions for controllability and observability (in the above
sense) are reported in \cite{tuna17} and \cite{tuna16},
respectively. Henceforth we assume:
\begin{center}
The array~\eqref{eqn:array} is both controllable and observable.
\end{center}

In the next section we present a distributed synchronization
algorithm that generates input signals $u_{ij}$ for the
array~\eqref{eqn:array} based on the measurements $y_{ij}$. This
algorithm is meant to achieve convergence
$\|x_{i}[k]-x_{j}[k]\|\to 0$ for all pairs $(i,\,j)$ and all
initial conditions.

\section{Algorithm}\label{sec:algorithm}

There are four design parameters to be chosen for the algorithm:
the integers $N_{\rm c},\,N_{\rm o}\geq n$ and the real numbers
$\tau_{\rm c},\,\tau_{\rm o}>0$. The variables employed are
denoted by $\lambda_{i},\,{\hat x}_{i},\,\xi_{i}\in\Real^{n}$, and
$w_{ij}\in(\Real^{p_{ij}})^{N_{\rm c}}$ for
$i,\,j=1,\,2,\,\ldots,\,q$. The variables ${\hat x}_{i}$ are
purely discrete and their values at the $k$th discrete time
instant is denoted by ${\hat x}_{i}[k]$. The remaining variables,
at each $k$, solve certain differential equations, the solutions
of which are denoted by $\lambda_{i}[k,\,t),\,\xi_{i}[k,\,t)$, and
$w_{ij}[k,\,t)$ with $t\in\Real$ being the continuous time
variable. Now, the algorithm generating the control inputs
$u_{ij}[k]$ for the array~\eqref{eqn:array} is as follows.
\begin{subeqnarray}\label{eqn:algorithm}
{\dot w}_{ij}[k,\,t)&=&-w_{ij}[k,\,t)-[B_{ij}\ AB_{ij}\ \cdots\ A^{N_{\rm c}-1}B_{ij}]^{T}(\lambda_{i}[k,\,t)-\lambda_{j}[k,\,t))\slabel{eqn:wij}\\
{\dot \lambda}_{i}[k,\,t)&=&A^{N_{\rm c}}{\hat x}_{i}[k]+\sum_{j=1}^{q}\,[B_{ij}\ AB_{ij}\ \cdots\ A^{N_{\rm c}-1}B_{ij}]w_{ij}[k,\,t)\slabel{eqn:lambdai}\\
{\hat x}_{i}[k+1]&=&A{\hat x}_{i}[k]+A^{N_{\rm o}}\xi_{i}[k,\,\tau_{\rm o})+\sum_{j=1}^{q}B_{ij}u_{ij}[k]\slabel{eqn:xhati}\\
{\dot \xi}_{i}[k,\,t)&=& -\sum_{\ell=0}^{N_{\rm
o}-1}\sum_{j=1}^{q}A^{\ell
T}C_{ij}^{T}C_{ij}A^{\ell}(\xi_{i}[k,\,t)-\xi_{j}[k,\,t))\nonumber\\
&&\qquad+\sum_{j=1}^{q}A^{(N_{\rm
o}-1)T}C_{ij}^{T}(y_{ij}[k]-C_{ij}({\hat x}_{i}[k]-{\hat x}_{j}[k]))\slabel{eqn:xii}\\
u_{ij}[k]&=&[\,\underbrace{[0\ \cdots\ 0\ 1]}_{\displaystyle
N_{\rm c}\,\mbox{terms}}\otimes I_{p_{ij}}]w_{ij}[k,\,\tau_{\rm
c})\slabel{eqn:uij}
\end{subeqnarray}
where, for all $k$, the initial conditions for integrations are
set as
\begin{eqnarray*}
\lambda_{i}[k,\,0)=0\,,\quad
\xi_{i}[k,\,0)=0\,,\quad
w_{ij}[k,\,0)=0\,.
\end{eqnarray*}
As for ${\hat x}_{i}$, the initial conditions ${\hat x}_{i}[0]$
can be chosen arbitrarily. Having described our algorithm, we can
now state what it does. Below is our main result.

\begin{theorem}\label{thm:main}
Consider the array~\eqref{eqn:array} under the control
inputs~\eqref{eqn:uij}. There exist real numbers $\bar\tau_{\rm
c}$ and $\bar\tau_{\rm o}$ such that if $\tau_{\rm
c}>\bar\tau_{\rm c}$ and $\tau_{\rm o}>\bar\tau_{\rm o}$, then the
systems synchronize, i.e., $\|x_{i}[k]-x_{j}[k]\|\to 0$ as
$k\to\infty$ for all pairs $(i,\,j)$ and all initial conditions
$x_{1}[0],\,x_{2}[0],\,\ldots,\,x_{q}[0]$.
\end{theorem}

In the remainder of the paper we construct the proof of
Theorem~\ref{thm:main}. To this end, we first obtain the
discrete-time closed-loop dynamics explicitly. Then we study its
stability.

\section{Closed loop}\label{sec:closedloop}

In this section we compute the closed-loop dynamics governing the
system~\eqref{eqn:bigsystem} under the
algorithm~\eqref{eqn:algorithm}. Namely, we obtain explicit
expressions for the matrices $\Kbig$ and $\Lbig$ which should
appear as
\begin{subeqnarray}\label{eqn:bigclosedloop}
\ex^{+}&=&\Abig\ex-\Bbig\Kbig\hat\ex\slabel{eqn:bigclosedloopA}\\
\hat\ex^{+}&=&\Abig\hat\ex-\Bbig\Kbig\hat\ex+{\mathbf
L}(\vay-\Cbig\hat\ex)\slabel{eqn:bigclosedloopB}
\end{subeqnarray}
where $\hat\ex=[\hat x_{1}^{T}\ \hat x_{2}^{T}\ \cdots\ \hat
x_{q}^{T}]^{T}$ and $\hat x_{i}$ are updated via
\eqref{eqn:xhati}. We begin with $\Kbig$.

\subsection{Gain $\Kbig$}

We denote by $e_{N_{\rm c}}\in\Real^{N_{\rm c}}$ the unit vector
whose last entry is one, i.e., $e_{N_{\rm c}}=[0\ \cdots\ 0\
1]^{T}$. Recall that the vector ${\bf 1}_{q}$ spans the
synchronization subspace $\setS_{1}$. Let $S$ denote its
normalization, i.e., $S={\bf 1}_{q}/\sqrt{q}$ and hence
$S^{T}S=1$. Also, let $D\in\Real^{q\times (q-1)}$ be some matrix
whose columns make an orthonormal basis for $\setS_{1}^{\perp}$.
Note that $D^{T}D=I_{q-1}$ and the columns of the matrix $[D\ S]$
make an orthonormal basis for $\Real^{n}$. We let $\Dbig=[D\otimes
I_{n}]$ and $\Sbig=[S\otimes I_{n}]$. The following identities are
easy to show and find use in the sequel.
\begin{enumerate}
\item[(i)]$\Dbig\Dbig^{T}+\Sbig\Sbig^{T}=I_{qn}$. \item[(ii)]
${\rm range}\,[D\otimes I_{n}]=\setS_{n}^{\perp}$. \item[(iii)]
$[S^{T}\otimes I_{n}]\Bbig=0$. \item[(iv)] $\Cbig[S\otimes
I_{n}]=0$.
\end{enumerate}

Recall that $w_{ij}\in(\Real^{p_{ij}})^{N_{\rm c}}$ are the variables
in \eqref{eqn:wij}. Note that $B_{ij}=-B_{ji}$ yields
$\dot{w}_{ij}+w_{ij}=\dot{w}_{ji}+w_{ji}$. Then
$w_{ij}[k,\,0)=w_{ji}[k,\,0)$ implies $w_{ij}[k,\,t)\equiv
w_{ji}[k,\,t)$. This allows us to consider in our analysis only
$w_{ij}$ with $i<j$. Now, let us partition $w_{ij}$ as
$w_{ij}=[w_{ij}^{[N_{\rm c}-1]T}\ w_{ij}^{[N_{\rm c}-2]T}\
\cdots\ w_{ij}^{[0]T}]^{T}$ with
$w_{ij}^{[\ell]}\in\Real^{p_{ij}}$. Then gather $w_{ij}^{[\ell]}$
as $\dabilyu^{[\ell]}=[w_{12}^{[\ell]T}\ w_{13}^{[\ell]T}\ \cdots\
w_{1q}^{[\ell]T}\,|\, w_{23}^{[\ell]T}\ w_{24}^{[\ell]T}\ \cdots\
w_{2q}^{[\ell]T}\,|\,\cdots\,|\,w_{(q-1)q}^{[\ell]T}]^{T}$ to
construct $\dabilyu=[\dabilyu^{[N_{\rm c}-1]T}\ \dabilyu^{[N_{\rm
c}-2]T}\ \cdots\ \dabilyu^{[0]T}]^{T}$. Also, let
$\lambda=[\lambda_{1}^{T}\ \lambda_{2}^{T}\ \cdots\
\lambda_{q}^{T}]^{T}$ where $\lambda_{i}\in\Real^{n}$ are the
variables in \eqref{eqn:lambdai}. Finally define
\begin{eqnarray*}
\Rbig=[\Bbig\ \Abig\Bbig\ \cdots\ \Abig^{N_{\rm c}-1}\Bbig]\,.
\end{eqnarray*}
This new set of notation allows us to put the
dynamics~\eqref{eqn:wij}-\eqref{eqn:lambdai} into the following
compact form
\begin{eqnarray}\label{eqn:compact}
\left[\begin{array}{c} \dot\dabilyu[k,\,t)\\
\dot\lambda[k,\,t)
\end{array}\right]=\left[\begin{array}{cc}
-I&-\Rbig^{T}\\
\Rbig & 0
\end{array}\right]\left[\begin{array}{c} \dabilyu[k,\,t)\\
\lambda[k,\,t)
\end{array}\right]+\left[\begin{array}{c} 0\\
\Abig^{N_{\rm c}}\hat\ex[k]
\end{array}\right]\,,\quad\dabilyu[k,\,0)=0\,,\quad
\lambda[k,\,0)=0\,.
\end{eqnarray}
Solving \eqref{eqn:compact} allows us to obtain the inputs
generated by the algorithm~\eqref{eqn:algorithm} because
$\yu[k]=\dabilyu^{[0]}[k,\,\tau_{\rm c})$. Since we are not
interested in the solution $\lambda[k,\,t)$ let us consider
another differential equation, which in certain ways is more
convenient:
\begin{eqnarray}\label{eqn:convenient}
\left[\begin{array}{c} \dot\vi[k,\,t)\\
\dot\mu[k,\,t)
\end{array}\right]=\underbrace{\left[\begin{array}{cc}
-I&-\Rbig^{T}\Dbig\\
\Dbig^{T}\Rbig & 0
\end{array}\right]}_{\displaystyle \Lambda}\left[\begin{array}{c} \vi[k,\,t)\\
\mu[k,\,t)
\end{array}\right]+\left[\begin{array}{c} 0\\
\Dbig^{T}\Abig^{N_{\rm c}}\hat\ex[k]
\end{array}\right]\,,\quad\vi[k,\,0)=0\,,\quad
\mu[k,\,0)=0
\end{eqnarray}
where the size of the vector $\vi$ is same as that of $\dabilyu$
and the vector $\mu$ is of appropriate size. We now make a
succession of simple observations that eventually lead us to an
explicit expression for the gain $\Kbig$.

\begin{lemma}\label{lem:initial}
We have $\Dbig\Dbig^{T}\Abig^{\ell}\Bbig=\Abig^{\ell}\Bbig$ for
any integer $\ell\geq 0$.
\end{lemma}

\begin{proof}
Observe that
\begin{eqnarray*}
\Sbig^{T}\Abig^{\ell}\Bbig
&=&[S^{T}\otimes I_{n}][I_{q}\otimes A]^{\ell}\Bbig\\
&=&[S^{T}\otimes I_{n}][I_{q}\otimes A^{\ell}]\Bbig\\
&=&A^{\ell}\underbrace{[S^{T}\otimes I_{n}]\Bbig}_{0}\\
&=&0\,.
\end{eqnarray*}
Therefore we can write
\begin{eqnarray*}
\Dbig\Dbig^{T}\Abig^{\ell}\Bbig
&=&\Dbig\Dbig^{T}\Abig^{\ell}\Bbig+\Sbig(\Sbig^{T}\Abig^{\ell}\Bbig)\\
&=&(\underbrace{\Dbig\Dbig^{T}+\Sbig\Sbig^{T}}_{I})\Abig^{\ell}\Bbig\\
&=&\Abig^{\ell}\Bbig\,.
\end{eqnarray*}
Hence the result.
\end{proof}

\begin{lemma}\label{lem:convenient}
Consider the differential equations \eqref{eqn:compact} and
\eqref{eqn:convenient}. We have $\dabilyu[k,\,t)=\vi[k,\,t)$.
\end{lemma}

\begin{proof}
Consider \eqref{eqn:compact}. We can write
\begin{eqnarray*}
\ddot\dabilyu &=& -\dot\dabilyu-\Rbig^{T}\dot\lambda\\
&=&-\dot\dabilyu-\Rbig^{T}(\Rbig\dabilyu+\Abig^{N_{\rm
c}}\hat\ex)\,.
\end{eqnarray*}
Also,
$\dot\dabilyu[k,\,0)=-\dabilyu[k,\,0)-\Rbig^{T}\lambda[k,\,0)=0$.
Hence the solution $t\mapsto\dabilyu[k,\,t)$ should satisfy
\begin{eqnarray}\label{eqn:component1}
{\ddot\dabilyu}[k,\,t)+{\dot\dabilyu}[k,\,t)+\Rbig^{T}\Rbig\dabilyu[k,\,t)+\Rbig^{T}\Abig^{N_{\rm
c}} \hat\ex[k]=0\,,\quad
\dabilyu[k,\,0)=0\,,\quad{\dot\dabilyu}[k,\,0)=0\,.
\end{eqnarray}
Similarly, \eqref{eqn:convenient} implies
\begin{eqnarray}\label{eqn:component2}
{\ddot\vi}[k,\,t)+{\dot\vi}[k,\,t)+\Rbig^{T}\Dbig\Dbig^{T}\Rbig\vi[k,\,t)+\Rbig^{T}\Dbig\Dbig^{T}\Abig^{N_{\rm
c}} \hat\ex[k]=0\,,\quad \vi[k,\,0)=0\,,\quad{\dot\vi}[k,\,0)=0\,.
\end{eqnarray}
Lemma~\ref{lem:initial} allows us to write
\begin{eqnarray}\label{eqn:component3}
\Rbig^{T}\Dbig\Dbig^{T} &=& \left(\Dbig\Dbig^{T}[\Bbig\
\Abig\Bbig\ \cdots\ \Abig^{N_{\rm
c}-1}\Bbig]\right)^{T}\nonumber\\
&=& [\Bbig\ \Abig\Bbig\ \cdots\ \Abig^{N_{\rm
c}-1}\Bbig]^{T}\nonumber\\
&=& \Rbig^{T}\,.
\end{eqnarray}
Combining \eqref{eqn:component1}, \eqref{eqn:component2}, and
\eqref{eqn:component3} yields the result.
\end{proof}

\begin{lemma}\label{lem:fullrow}
The matrix $\Dbig^{T}\Rbig$ is full row rank.
\end{lemma}

\begin{proof}
Suppose not. Then we can find a nonzero vector
$\eta\in(\Real^{n})^{q-1}$ satisfying $\eta^{T}\Dbig^{T}\Rbig=0$.
Let $\zeta=\Dbig\eta$, which belongs to $\setS_{n}^{\perp}$ due to
${\rm range}\,[D\otimes I_{n}]=\setS_{n}^{\perp}$. Also,
$\zeta\neq 0$ because $\eta\neq 0$ and $\Dbig$ is full column
rank. Thence $\zeta\notin\setS_{n}$. This implies ${\rm
null}\,\Rbig^{T}\not\subset\setS_{n}$ due to $\Rbig^{T}\zeta=0$.
Hence ${\rm range}\,\Rbig\not\supset\setS_{n}^{\perp}$.
Consequently ${\rm range}\,\Wbig_{\rm
c}\not\supset\setS_{n}^{\perp}$ because ${\rm
range}\,\Rbig\supset{\rm range}\,\Wbig_{\rm c}$ thanks to $N_{\rm
c}\geq n$. But ${\rm range}\,\Wbig_{\rm
c}\not\supset\setS_{n}^{\perp}$ contradicts that the
array~\eqref{eqn:array} is controllable.
\end{proof}

\begin{lemma}\label{lem:hurwitz}
The matrix $\Lambda$ defined in \eqref{eqn:convenient} is Hurwitz,
i.e, all its eigenvalues are on the open left half-plane.
\end{lemma}

\begin{proof}
It is easy to see that $\Lambda^{T}+\Lambda\leq 0$. Therefore
$\Lambda$ is at least neutrally stable. In particular, it cannot
have any eigenvalues with positive real part. To show that it can
neither have any eigenvalues on the imaginary axis let us suppose
the contrary. That is, assume $j\omega$ with $\omega\in\Real$ is
an eigenvalue of $\Lambda$. Then we could find two vectors
$v_{1},\,v_{2}$, at least one of them nonzero, satisfying
\begin{eqnarray*}
\left[\begin{array}{cc}-I&-\Rbig^{T}\Dbig\\
\Dbig^{T}\Rbig&0\end{array}\right]\left[\begin{array}{c}v_{1}\\
v_{2}\end{array}\right]=j\omega\left[\begin{array}{c}v_{1}\\
v_{2}\end{array}\right]
\end{eqnarray*}
which yields $v_{1}=-(1+j\omega)^{-1}\Rbig^{T}\Dbig v_{2}$ and
$\Dbig^{T}\Rbig v_{1}=j\omega v_{2}$. Note that $v_{2}$ cannot be
zero, for otherwise $v_{1}$ would also have to be zero and by
assumption it cannot be that both are zero. Hence we combine the
two equations and write
\begin{eqnarray*}
[\Dbig^{T}\Rbig\Rbig^{T}\Dbig]v_{2}=-\frac{j\omega}{1+j\omega}\,v_{2}\,.
\end{eqnarray*}
That is, $v_{2}$ is an eigenvector of
$\Dbig^{T}\Rbig\Rbig^{T}\Dbig$. Since
$\Dbig^{T}\Rbig\Rbig^{T}\Dbig$ is a real symmetric matrix, its
eigenvalues are real. Therefore we have to have $\omega=0$. Thence
$[\Dbig^{T}\Rbig\Rbig^{T}\Dbig]v_{2}=0$, i.e.,
$\Dbig^{T}\Rbig\Rbig^{T}\Dbig$ is singular, which however cannot
be true because $\Dbig^{T}\Rbig$ is full row rank by
Lemma~\ref{lem:fullrow}. Hence $\Lambda$ has no eigenvalue on the
imaginary axis, which completes the proof.
\end{proof}

\begin{lemma}\label{lem:K}
The matrix $\Kbig$ in the closed-loop
system~\eqref{eqn:bigclosedloop} reads
\begin{eqnarray*}
\Kbig&=&\left[\Bbig^{T}\Abig^{(N_{\rm
c}-1)T}\Dbig[\Dbig^{T}\Rbig\Rbig^{T}\Dbig]^{-1}\Dbig^{T}\Rbig-[e_{N_{\rm
c}}^{T}\otimes I]\ \ \ \Bbig^{T}\Abig^{(N_{\rm
c}-1)T}\Dbig[\Dbig^{T}\Rbig\Rbig^{T}\Dbig]^{-1}\right]\\
&&\qquad\times[I-e^{\Lambda\tau_{\rm
c}}]\times\left[\begin{array}{c}0\\\Dbig^{T}\Abig^{N_{\rm
c}}\end{array}\right]\,.
\end{eqnarray*}
\end{lemma}

\begin{proof}
Consider \eqref{eqn:convenient}. It can be verified by direct
substitution that the solution reads
\begin{eqnarray*}
\left[\begin{array}{c}\vi[k,\,t)\\ \mu[k,\,t)\end{array}\right]
&=&\Lambda^{-1}[e^{\Lambda t}-I]\left[\begin{array}{c}0\\ {\Dbig}^{T}\Abig^{N_{\rm c}} \hat\ex[k]\end{array}\right]\\
&=&\left[\begin{array}{cc}\Rbig^{T}{\Dbig}[{\Dbig}^{T}{\Rbig}{\Rbig}^{T}{\Dbig}]^{-1}{\Dbig}^{T}{\Rbig}-I&{\Rbig}^{T}{\Dbig}
[{\Dbig}^{T}{\Rbig}{\Rbig}^{T}{\Dbig}]^{-1}\\
-[{\Dbig}^{T}{\Rbig}{\Rbig}^{T}{\Dbig}]^{-1}{\Dbig}^{T}{\Rbig}&-[{\Dbig}^{T}{\Rbig}{\Rbig}^{T}{\Dbig}]^{-1}\end{array}\right][e^{\Lambda
t}-I]\left[\begin{array}{c}0\\
\Dbig^{T}\Abig^{N_{\rm c}}\end{array}\right]\hat\ex[k]
\end{eqnarray*}
where $\Lambda^{-1}$ exists because $\Lambda$ is Hurwitz by
Lemma~\ref{lem:hurwitz} and
$[{\Dbig}^{T}{\Rbig}{\Rbig}^{T}{\Dbig}]^{-1}$ exists because
${\Dbig}^{T}{\Rbig}$ is full row rank by Lemma~\ref{lem:fullrow}.
Using $\yu[k]=\dabilyu^{[0]}[k,\,\tau_{\rm c})$ and
Lemma~\ref{lem:convenient} we can now write
\begin{eqnarray*}
\Kbig\hat\ex[k]
&=& -\yu[k]\\
&=& -\dabilyu^{[0]}[k,\,\tau_{\rm c})\\
&=& -[e_{N_{\rm c}}^{T}\otimes I]\dabilyu[k,\,\tau_{\rm c})\\
&=& -[e_{N_{\rm c}}^{T}\otimes I]\vi[k,\,\tau_{\rm c})\\
&=& [e_{N_{\rm c}}^{T}\otimes
I]\left[\Rbig^{T}{\Dbig}[{\Dbig}^{T}{\Rbig}{\Rbig}^{T}{\Dbig}]^{-1}{\Dbig}^{T}{\Rbig}-I\
\ \ {\Rbig}^{T}{\Dbig}
[{\Dbig}^{T}{\Rbig}{\Rbig}^{T}{\Dbig}]^{-1}\right][I-e^{\Lambda
\tau_{\rm c}}]\left[\begin{array}{c}0\\
\Dbig^{T}\Abig^{N_{\rm c}}\end{array}\right]\hat\ex[k]\,.
\end{eqnarray*}
The result then follows since $\Rbig[e_{N_{\rm c}}\otimes
I]=\Abig^{N_{\rm c}-1}\Bbig$.
\end{proof}

\subsection{Gain $\Lbig$}

Having computed $\Kbig$ of \eqref{eqn:bigclosedloop}, we now focus
on $\Lbig$. Let $\xi=[\xi_{1}^{T}\ \xi_{2}^{T}\ \cdots\
\xi_{q}^{T}]^{T}$, where $\xi_{i}\in\Real^{n}$ are the variables
in \eqref{eqn:xii}, and define
\begin{eqnarray*}
\Qbig=\left[\begin{array}{c}\Cbig\\ \Cbig\Abig\\ \vdots\\
\Cbig\Abig^{N_{\rm o}-1}\end{array}\right]\,.
\end{eqnarray*}
This allows the dynamics~\eqref{eqn:xii} to be compactly expressed
as
\begin{eqnarray}\label{eqn:comfort}
\dot\xi[k,\,t)=-\Qbig^{T}\Qbig\xi[k,\,t)+\Abig^{(N_{\rm
o}-1)T}\Cbig^{T}(\vay[k]-\Cbig\hat\ex[k])\,,\quad\xi[k,\,0)=0\,.
\end{eqnarray}
Finally, we define $\Gamma=\Dbig^{T}\Qbig^{T}\Qbig\Dbig$. Let us
make a few observations before we attempt to solve
\eqref{eqn:comfort}.

\begin{lemma}\label{lem:fullrow2}
The matrix $\Dbig^{T}\Qbig^{T}$ is full row rank.
\end{lemma}

\begin{proof}
This is a consequence of the observability of the
array~\eqref{eqn:array}. See the dual result
Lemma~\ref{lem:fullrow}.
\end{proof}

\begin{lemma}\label{lem:hurwitz2}
The matrix $-\Gamma$ is Hurwitz.
\end{lemma}

\begin{proof}
The matrix $\Gamma$ is symmetric positive semidefinite because we
can write $\Gamma=(\Qbig\Dbig)^{T}\Qbig\Dbig$. Also, it is
nonsingular because $(\Qbig\Dbig)^{T}$ is full row rank by
Lemma~\ref{lem:fullrow2}. Hence $\Gamma$ is symmetric positive
definite. Then $-\Gamma$ is symmetric negative definite and
consequently all its eigenvalues are real and strictly negative.
Hence the result.
\end{proof}

\begin{lemma}\label{lem:initial2}
We have $\Dbig\Dbig^{T}\Abig^{\ell T}\Cbig^{T}=\Abig^{\ell
T}\Cbig^{T}$ for any integer $\ell\geq 0$.
\end{lemma}

\begin{proof}
Like Lemma~\ref{lem:initial}.
\end{proof}

\begin{lemma}\label{lem:xi}
The solution to \eqref{eqn:comfort} reads
$\xi[k,\,t)=\Dbig\Gamma^{-1}[I-e^{-\Gamma
t}]\Dbig^{T}\Abig^{(N_{\rm
o}-1)T}\Cbig^{T}(\vay[k]-\Cbig\hat\ex[k])$.
\end{lemma}

\begin{proof}
First note that the initial condition constraint $\xi[k,\,0)=0$ is
satisfied. Now we show that this $\xi[k,\,t)$ also satisfies the
differential equation. For compactness let $\beta=\Abig^{(N_{\rm
o}-1)T}\Cbig^{T}(\vay[k]-\Cbig\hat\ex[k])$. Note that
$\Dbig\Dbig^{T}\beta=\beta$ and
$\Dbig\Dbig^{T}\Qbig^{T}=\Qbig^{T}$ by Lemma~\ref{lem:initial2}.
Hence putting our candidate solution into the righthand side of
\eqref{eqn:comfort} yields
\begin{eqnarray*}
-\Qbig^{T}\Qbig\xi[k,\,t)+\beta&=&-\Qbig^{T}\Qbig\Dbig\Gamma^{-1}[I-e^{-\Gamma
t}]\Dbig^{T}\beta+\beta\\
&=&-\Dbig\underbrace{\Dbig^{T}\Qbig^{T}\Qbig\Dbig}_{\Gamma}\Gamma^{-1}[I-e^{-\Gamma
t}]\Dbig^{T}\beta+\beta\\
&=&-\Dbig[I-e^{-\Gamma
t}]\Dbig^{T}\beta+\beta\\
&=&\Dbig e^{-\Gamma
t}\Dbig^{T}\beta-\Dbig\Dbig^{T}\beta+\beta\\
&=&\Dbig e^{-\Gamma t}\Dbig^{T}\beta\,.
\end{eqnarray*}
Now, by differentiating the candidate solution we obtain
\begin{eqnarray*}\label{eqn:willy}
\dot\xi[k,\,t)&=&\frac{d}{dt}\left\{\Dbig\Gamma^{-1}[I-e^{-\Gamma t}]\Dbig^{T}\beta\right\}\nonumber\\
&=&\Dbig\Gamma^{-1}[\Gamma e^{-\Gamma t}]\Dbig^{T}\beta\nonumber\\
&=&\Dbig e^{-\Gamma t}\Dbig^{T}\beta\\
&=&-\Qbig^{T}\Qbig\xi[k,\,t)+\beta
\end{eqnarray*}
which was to be shown.
\end{proof}

\begin{lemma}\label{lem:L}
The matrix $\Lbig$ in the closed-loop
system~\eqref{eqn:bigclosedloop} reads
\begin{eqnarray*}
\Lbig=\Abig^{N_{\rm o}}\Dbig\Gamma^{-1}[I-e^{-\Gamma \tau_{\rm
o}}]\Dbig^{T}\Abig^{(N_{\rm o}-1)T}\Cbig^{T}\,.
\end{eqnarray*}
\end{lemma}

\begin{proof}
Using \eqref{eqn:xhati} and Lemma~\ref{lem:xi} we can write
\begin{eqnarray*}
\hat\ex[k+1]
&=&\Abig\hat\ex[k]-\Bbig{\mathbf
K}\hat\ex[k]+\Abig^{N_{\rm o}}\xi[k,\,\tau_{\rm o})\\
&=&\Abig\hat\ex[k]-\Bbig\Kbig\hat\ex[k]+\Abig^{N_{\rm
o}}\Dbig\Gamma^{-1}[I-e^{-\Gamma \tau_{\rm
o}}]\Dbig^{T}\Abig^{(N_{\rm
o}-1)T}\Cbig^{T}(\vay[k]-\Cbig\hat\ex[k])\,.
\end{eqnarray*}
Comparing this to \eqref{eqn:bigclosedloopB} yields the result.
\end{proof}

\section{Stability}\label{sec:stability}

In the previous section we obtained the explicit expression for
the righthand side of \eqref{eqn:bigclosedloop} by computing the
gains $\Kbig$ and $\Lbig$. Now we study the behavior of the
closed-loop system. Our goal is to show that (under appropriate
choices of the parameters $\tau_{\rm c}, \tau_{\rm o}$) for all
initial conditions $\ex[0]$ and $\hat\ex[0]$ the solution $\ex[k]$
enjoys the convergence $\|\ex[k]\|_{\setS_{n}}\to 0$, where
$\|\cdot\|_{\setS_{n}}$ denotes the Euclidean distance to the
subspace $\setS_{n}$. By proving this convergence we will have
established Theorem~\ref{thm:main}. For our analysis we borrow the
following result due to
Kleinman~\cite{Kleinman74}.\footnote{Kleinman assumes that the
matrix $F$ is invertible. This assumption however is superfluous;
see \cite{Tuna15}.}

\begin{proposition}\label{prop:key}
Let $N\geq 1$ be an integer, $F\in\Real^{n\times n}$,
$G\in\Real^{n\times p}$, and ${\rm rank}\,[G\ FG\ \cdots\
F^{N-1}G]=n$. Then the matrix
\begin{eqnarray*}
H=F-GG^{T}F^{(N-1)T}\left(\sum_{\ell=0}^{N-1}F^{\ell}GG^{T}F^{\ell
T}\right)^{-1}F^{N}
\end{eqnarray*}
is Schur, i.e., all the eigenvalues of $H$ are on the open unit
disc.
\end{proposition}

Define the reduced state $\exr=\Dbig^{T}\ex$ and the error
$\er=\Dbig^{T}(\hat\ex-\ex)$. Note that
$\|\exr\|=\|\ex\|_{\setS_{n}}$. Also, define the following reduced
parameters
\begin{eqnarray*}
\Abigr=\Dbig^{T}\Abig\Dbig\,,\quad \Bbigr=\Dbig^{T}\Bbig\,,\quad
\Cbigr=\Cbig\Dbig\,,\quad \Kbigr=\Kbig\Dbig\,,\quad
\Lbigr=\Dbig^{T}\Lbig\,.
\end{eqnarray*}

\begin{lemma}\label{lem:initial3}
We have
$\Dbig^{T}\Abig^{\ell}=\Dbig^{T}\Abig^{\ell}\Dbig\Dbig^{T}$ for
any integer $\ell\geq 0$.
\end{lemma}

\begin{proof}
Like Lemma~\ref{lem:initial}.
\end{proof}
\vspace{0.12in}

Note that Lemma~\ref{lem:K} and Lemma~\ref{lem:initial3} imply
$\Kbig=\Kbig\Dbig\Dbig^{T}$. Using the structural properties of
our matrices $\Abig,\,\Bbig,\,\Cbig$ emphasized in
Lemmas~\ref{lem:initial}, \ref{lem:initial2}, and
\ref{lem:initial3}, we now proceed to obtain the dynamics for
$\exr$ and $\er$. Consider \eqref{eqn:bigclosedloopA}. We can
write
\begin{eqnarray*}
\exr^{+}&=&\Dbig^{T}\ex^{+}\\
&=&\Dbig^{T}\Abig\ex-\Dbig^{T}\Bbig\Kbig\hat\ex\\
&=&\Dbig^{T}\Abig\Dbig\Dbig^{T}\ex-\Dbig^{T}\Bbig\Kbig\Dbig\Dbig^{T}\hat\ex\\
&=&\Abigr\exr-\Bbigr\Kbigr(\exr+\er)\\
&=&[\Abigr-\Bbigr\Kbigr]\exr-\Bbigr\Kbigr\er\,.
\end{eqnarray*}
As for $\er$, the dynamics~\eqref{eqn:bigclosedloop} yields
\begin{eqnarray*}
\er^{+}&=&\Dbig^{T}(\hat\ex^{+}-\ex^{+})\\
&=&\Dbig^{T}(\Abig\hat\ex+\Lbig(\Cbig\ex-\Cbig\hat\ex)-\Abig\ex)\\
&=&[\Dbig^{T}\Abig-\Dbig^{T}\Lbig\Cbig](\hat\ex-\ex)\\
&=&[\Dbig^{T}\Abig\Dbig\Dbig^{T}-\Dbig^{T}\Lbig\Cbig\Dbig\Dbig^{T}](\hat\ex-\ex)\\
&=&[\Dbig^{T}\Abig\Dbig-\Dbig^{T}\Lbig\Cbig\Dbig]\Dbig^{T}(\hat\ex-\ex)\\
&=&[\Abigr-\Lbigr\Cbigr]\er\,.
\end{eqnarray*}
Hence the overall dynamics for the pair $(\exr,\,\er)$ reads
\begin{eqnarray}\label{eqn:reduced}
\left[\begin{array}{c}\exr\\ \er\end{array}\right]^{+}=
\underbrace{\left[\begin{array}{cc}\Abigr-\Bbigr\Kbigr & -\Bbigr\Kbigr\\
0 & \Abigr-\Lbigr\Cbigr\end{array}\right]}_{\displaystyle {\mathbf
\Phi}_{\rm r}}
\left[\begin{array}{c}\exr\\
\er\end{array}\right]\,.
\end{eqnarray}
Next, we show that the block diagonal entries in
\eqref{eqn:reduced} can be made Schur by choosing $\tau_{\rm c}$
and $\tau_{\rm o}$ large enough. To this end, let us define the
following $(q-1)n\times(q-1)n$ matrices.
\begin{eqnarray*}
\theta_{\rm c}(\tau)&=&\Dbig^{T}\Bbig\left[\Bbig^{T}\Abig^{(N_{\rm
c}-1)T}\Dbig[\Dbig^{T}\Rbig\Rbig^{T}\Dbig]^{-1}\Dbig^{T}\Rbig-[e_{N_{\rm
c}}^{T}\otimes I]\ \ \ \Bbig^{T}\Abig^{(N_{\rm
c}-1)T}\Dbig[\Dbig^{T}\Rbig\Rbig^{T}\Dbig]^{-1}\right]\\
&&\qquad\times
[e^{\Lambda\tau}]\times\left[\begin{array}{c}0\\\Dbig^{T}\Abig^{N_{\rm
c}}\end{array}\right]\Dbig\\
\theta_{\rm o}(\tau)&=&\Dbig^{T}\Abig^{N_{\rm
o}}\Dbig\Gamma^{-1}[e^{-\Gamma \tau}]\Dbig^{T}\Abig^{(N_{\rm
o}-1)T}\Cbig^{T}\Cbig\Dbig\,.
\end{eqnarray*}
Now we can write by Lemmas~\ref{lem:initial}, \ref{lem:K}, and
\ref{lem:initial3}
\begin{eqnarray}\label{eqn:schurc}
\Abigr-\Bbigr\Kbigr&=&\Abigr-\Dbig^{T}\Bbig\Kbig\Dbig\nonumber\\
&=&\Abigr-\Dbig^{T}\Bbig\Bbig^{T}\Abig^{(N_{\rm
c}-1)T}\Dbig[\Dbig^{T}\Rbig\Rbig^{T}\Dbig]^{-1}\Dbig^{T}\Abig^{N_{\rm
c}}\Dbig+\theta_{\rm c}(\tau_{\rm c})\nonumber\\
&=&\Abigr-\Bbigr\Bbigr^{T}\Abigr^{(N_{\rm
c}-1)T}[\Dbig^{T}\Rbig\Rbig^{T}\Dbig]^{-1}\Abigr^{N_{\rm
c}}+\theta_{\rm c}(\tau_{\rm c})\nonumber\\
&=&\underbrace{\Abigr-\Bbigr\Bbigr^{T}\Abigr^{(N_{\rm
c}-1)T}\left(\sum_{\ell=0}^{N_{\rm
c}-1}\Abigr^{\ell}\Bbigr\Bbigr^{T}\Abigr^{\ell
T}\right)^{-1}\Abigr^{N_{\rm c}}}_{\displaystyle \Hbig_{\rm
c}}+\,\theta_{\rm c}(\tau_{\rm c})
\end{eqnarray}
where we used $\Dbig^{T}\Rbig=[\Bbigr\ \Abigr\Bbigr\ \cdots\
\Abigr^{N_{\rm c}-1}\Bbigr]$ and $\Abigr^{\ell}=\Dbig^{T}\Abig^{\ell}\Dbig$. Similarly, by
Lemmas~\ref{lem:initial2}, \ref{lem:L}, and \ref{lem:initial3} we obtain
\begin{eqnarray}\label{eqn:schuro}
\Abigr-\Lbigr\Cbigr&=&\Abigr-\Dbig^{T}\Lbig\Cbig\Dbig\nonumber\\
&=&\Abigr-\Dbig^{T}\Abig^{N_{\rm
o}}\Dbig\Gamma^{-1}\Dbig^{T}\Abig^{(N_{\rm
o}-1)T}\Cbig^{T}\Cbig\Dbig+\theta_{\rm o}(\tau_{\rm o})\nonumber\\
&=&\Abigr-\Abigr^{N_{\rm
o}}[\Dbig^{T}\Qbig^{T}\Qbig\Dbig]^{-1}\Abigr^{(N_{\rm
o}-1)T}\Cbigr^{T}\Cbigr+\theta_{\rm o}(\tau_{\rm o})\nonumber\\
&=&\underbrace{\Abigr-\Abigr^{N_{\rm
o}}\left(\sum_{\ell=0}^{N_{\rm o}-1}\Abigr^{\ell
T}\Cbigr^{T}\Cbigr\Abigr^{\ell}\right)^{-1}\Abigr^{(N_{\rm
o}-1)T}\Cbigr^{T}\Cbigr}_{\displaystyle \Hbig_{\rm
o}}+\,\theta_{\rm o}(\tau_{\rm o})
\end{eqnarray}
where we used $\Dbig^{T}\Qbig^{T}=[\Cbigr^{T}\
\Abigr^{T}\Cbigr^{T}\ \cdots\ \Abigr^{(N_{\rm o}-1)T}\Cbigr^{T}]$.

\begin{lemma}\label{lem:schurc}
There exists $\bar\tau>0$ such that the matrix $[\Hbig_{\rm
c}+\theta_{\rm c}(\tau)]$ is Schur for all $\tau>\bar\tau$.
\end{lemma}

\begin{proof}
By Lemma~\ref{lem:fullrow} the matrix $[\Bbigr\ \Abigr\Bbigr\
\cdots\ \Abigr^{N_{\rm c}-1}\Bbigr]=\Dbig^{T}\Rbig$ is full row
rank. Then $\Hbig_{\rm c}$ is Schur by Proposition~\ref{prop:key}.
Recall that $\Lambda$ is Hurwitz by Lemma~\ref{lem:hurwitz}.
Therefore $\theta_{\rm c}(\tau)\to 0$ because $e^{\Lambda \tau}\to
0$ as $\tau\to\infty$. Now, since small perturbations on a square
matrix mean small changes on its eigenvalues, we can find
$\varepsilon>0$ such that $[\Hbig_{\rm c}+\theta]$ is Schur for
all $\|\theta\|<\varepsilon$. Then, thanks to $\theta_{\rm
c}(\tau)\to 0$, we can choose some $\bar\tau>0$ such that
$\|\theta_{\rm c}(\tau)\|<\varepsilon$ for all $\tau>\bar\tau$.
Hence the result.
\end{proof}

\begin{lemma}\label{lem:schuro}
There exists $\bar\tau>0$ such that the matrix $[\Hbig_{\rm
o}+\theta_{\rm o}(\tau)]$ is Schur for all $\tau>\bar\tau$.
\end{lemma}

\begin{proof}
By Lemma~\ref{lem:fullrow2} the matrix $[\Cbigr^{T}\
\Abigr\Cbigr^{T}\ \cdots\ \Abigr^{(N_{\rm
o}-1)T}\Cbigr^{T}]=\Dbig^{T}\Qbig^{T}$ is full row rank. Then
$\Hbig_{\rm o}^{T}$ is Schur by Proposition~\ref{prop:key},
meaning $\Hbig_{\rm o}$ is also Schur. Lemma~\ref{lem:hurwitz2}
tells us that $-\Gamma$ is Hurwitz. Hence $e^{-\Gamma \tau}\to 0$
and, consequently, $\theta_{\rm o}(\tau)\to 0$ as $\tau\to\infty$.
The rest is like the proof of Lemma~\ref{lem:schurc}.
\end{proof}
\vspace{0.12in}

Our preparations for the proof of Theorem~\ref{thm:main} are now
complete: \vspace{0.12in}

\noindent{\bf Proof of Theorem~\ref{thm:main}.} Consider the
array~\eqref{eqn:array} under the control inputs~\eqref{eqn:uij}
and arbitrary initial conditions
$x_{1}[0],\,x_{2}[0],\,\ldots,\,x_{q}[0]$. Let $\bar\tau_{\rm
c}>0$ be such that $[\Hbig_{\rm c}+\theta_{\rm c}(\tau)]$ is Schur
for all $\tau>\bar\tau_{\rm c}$. Likewise, let $\bar\tau_{\rm
o}>0$ be such that $[\Hbig_{\rm o}+\theta_{\rm o}(\tau)]$ is Schur
for all $\tau>\bar\tau_{\rm o}$. Such $\bar\tau_{\rm c}$ and
$\bar\tau_{\rm o}$ exist thanks, respectively, to
Lemma~\ref{lem:schurc} and Lemma~\ref{lem:schuro}. Suppose now the
parameters $\tau_{\rm c}$ and $\tau_{\rm o}$ in the
algorithm~\eqref{eqn:algorithm} satisfy $\tau_{\rm
c}>\bar\tau_{\rm c}$ and $\tau_{\rm o}>\bar\tau_{\rm o}$. Then, by
\eqref{eqn:schurc} the matrix
$[\Abigr-\Bbigr\Kbigr]$ is Schur. Likewise,
$[\Abigr-\Lbigr\Cbigr]$ is Schur by \eqref{eqn:schuro}. Therefore the system matrix ${\mathbf
\Phi}_{\rm r}$ in \eqref{eqn:reduced} is Schur because it is upper
block triangular with Schur block diagonal entries. This implies
that the solution of the system~\eqref{eqn:reduced} must converge
to the origin regardless of the initial conditions. In particular,
we have $\exr[k]\to 0$ as $k\to\infty$. Then the solution of the
closed-loop system~\eqref{eqn:bigclosedloop} must satisfy
$\|\ex[k]\|_{\setS_{n}}\to 0$ because
$\|\ex\|_{\setS_{n}}=\|\Dbig^{T}\ex\|=\|\exr\|$. Clearly,
$\|\ex[k]\|_{\setS_{n}}\to 0$ means $\|x_{i}[k]-x_{j}[k]\|\to 0$
as $k\to\infty$ for all pairs $(i,\,j)$. Hence the result.
\hfill$\blacksquare$

\section{Conclusion}

In this paper we studied relatively actuated arrays of
discrete-time linear time-invariant systems with incommensurable
coupling parameters. For this general class of arrays we presented
a distributed algorithm that achieved synchronization through
dynamic output feedback. In the case we studied, even though the
array evolved in discrete time, part of the algorithm required
integration in continuous time so that the overall process
remained decentralized. At this point it is not clear how to
construct a purely discrete-time algorithm for a discrete-time
array. As for continuous-time arrays with incommensurable
input/output matrices, the problem of distributed synchronization
under relative actuation seems still to be open.

\bibliographystyle{plain}
\bibliography{references}
\end{document}